\documentclass[10]{amsart}
\usepackage{latexsym,amssymb,amsmath,amsthm,amsopn,graphics,xy,epsfig,picture,epic}
\usepackage{color}
\def\bb{{\mathcal B}}

\def\ff{{\mathcal F}}

\def\ll{{\mathcal L}}

\def\tt{{\mathcal T}}

\def\ffi{\varphi}
\def\eps{\varepsilon}
\def\dst{\displaystyle}

\def\B{{\mathbb{B}}}
\def\C{{\mathbb{C}}}

\def\R{{\mathbb{R}}}

\def\Z{{\mathbb{Z}}}

\newcommand{\norm}[1]{{\left\|{#1}\right\|}}
\newcommand{\ent}[1]{{\left[{#1}\right]}}
\newcommand{\abs}[1]{{\left|{#1}\right|}}
\newcommand{\scal}[1]{{\left\langle{#1}\right\rangle}}

\newenvironment{notation}[1][]{\vskip1pt\noindent\rm\textit{Notation}\,:\ }{\rm\vskip1pt}
\newenvironment{definition}[1][]{\vskip3pt\noindent\sl\textbf{Definition.}\ }{\rm\vskip3pt}
\newenvironment{remark}[1][]{\vskip3pt\noindent\textbf{Remark.}\ }{\rm\vskip3pt}

\newtheorem{lemma}{Lemma}[section]
\newtheorem{proposition}[lemma]{Proposition}
\newtheorem{theorem}[lemma]{Theorem}
\newtheorem{corollary}[lemma]{Corollary}

\date{\today}

\newcounter{rep}
\newcommand{\rep}[1]{
		}

\newcounter{rea}
\newcommand{\rea}[1]{
		}

\newcounter{res}
\begin{document}

\title[Almost time and band limited functions]{The approximation of almost time and band limited functions by their
expansion in some orthogonal polynomials bases}

\author{Philippe Jaming, Abderrazek Karoui, Susanna Spektor}

\address{Philippe Jaming
\noindent Address: Institut de Math\'ematiques de Bordeaux UMR 5251,
Universit\'e Bordeaux 1, cours de la Lib\'eration, F 33405 Talence cedex, France}
\email{Philippe.Jaming@gmail.com}

\address{Abderrazek Karoui
\noindent Address: Universit\'e de Carthage,
D\'epartement de Math\'ematiques, Facult\'e des Sciences de
Bizerte, Tunisie.}
\email{Abderrazek.Karoui@fsb.rnu.tn}

\address{Susanna Spektor
\noindent Address: Department of Mathematics, Michigan State University, 619 Red Cedar Road, East Lancing, MI 48824}
\email{sanaspek@gmail.com}

\begin{abstract}
The aim of this paper is to investigate the quality of approximation of almost time and almost band-limited
functions by its expansion in three classical orthogonal polynomials bases: the  Hermite, Legendre and Chebyshev bases. As a  corollary,
this allows us to obtain the quality of approximation in the $L^2-$Sobolev space by these orthogonal polynomials  bases. Also,  
we obtain the rate of the Legendre series expansion of the prolate spheroidal wave functions. Some numerical examples 
are given to illustrate the different results of this work. 
\end{abstract}




\maketitle

\section{Introduction}

Time-limited functions and band-limited functions play a fundamental role in signal and image processing.
The time-limiting assumption is natural as a signal can only be  measured over a finite duration.
The band-limiting assumption is natural as well due to channel capacity limitations. It is also
essential to apply sampling theory. Unfortunately, the simplest form of the uncertainty principle
tells us that a signal can not be simultaneously time and band limited.
A natural assumption is thus that a signal is almost time- and almost band-limited in the following sense:

\begin{definition} Let $T,\Omega>0$ and $\eps_T,\eps_\Omega>0$. A function $f\in L^2(\R)$
is said to be
\begin{itemize}
\item $\eps_T$-\emph{almost time limited} to $[-T,T]$ if
$$
\int_{|t|>T}|f(t)|^2\,\mathrm{d}t\leq\eps_T^2\norm{f}_{L^2(\R)}^2;
$$
\item $\eps_\Omega$-\emph{almost band limited} to $[-\Omega,\Omega]$ if
$$
\int_{|\omega|>\Omega}|\widehat{f}(\omega)|^2\,\mathrm{d}\omega\leq\eps_\Omega^2\norm{f}_{L^2(\R)}^2.
$$
\end{itemize}
Here and throughout this paper the Fourier transform is normalized so that, for $f\in L^1(\R)$,
$$
\widehat{f}(\omega):=\ff[f](\omega):=\frac{1}{\sqrt{2\pi}}\int_{\R}f(t)e^{-it\omega}\,\mathrm{d}t.
$$
\end{definition}

Of course, given $f\in L^2(\R)$, for every $\eps_T,\eps_\Omega>0$ there exist $T,\Omega>0$ such that
$f$ is $\eps_T$-almost time limited to $[-T,T]$ and $\eps_\Omega$-almost time limited to $[-\Omega,\Omega]$.
The point here is that we consider $T,\Omega,\eps_T,\eps_\Omega$ as fixed parameters.
A typical example we have in mind is that $f\in H^s(\R)$ and is time-limited to $[-T,T]$.
Such  an hypothesis is common in tomography, {\it see e.g.} \cite{Nat},
where it is required in the proof of the convergence of the filtered back-projection algorithm for approximate
inversion of the Radon transform.
But, if $f\in H^s(\R)$ with $s>0$, that is if
$$
\norm{f}_{H^s(\R)}^2:=\int_{\R}(1+|\omega|)^{2s}|\widehat{f}(\omega)|^2\,\mbox{d}\omega<+\infty,
$$
then
\begin{eqnarray*}
\int_{|\omega|>\Omega}|\widehat{f}(\omega)|^2\,\mbox{d}\omega&\leq&
\int_{|\omega|>\Omega}\frac{(1+|\omega|)^{2s}}{(1+|\Omega|)^{2s}}|\widehat{f}(\omega)|^2\,\mbox{d}\omega\\
&\leq&\frac{\norm{f}_{H^s(\R)}^2}{(1+|\Omega|)^{2s}}.
\end{eqnarray*}
Thus $f$ is $\dst\frac{1}{(1+|\Omega|)^{s}}\frac{\norm{f}_{H^s}}{\norm{f}_{L^2(\R)}}$-almost band limited to
$[-\Omega,\Omega]$.

An alternative to the back projection algorithms in tomography are the Algebraic Reconstruction Techniques
(that is variants of Kaczmarz algorithm, {\it see} \cite{Nat}). For those algorithms to work well it is crucial to have a good
representing system (basis, frame...) of the functions that one wants to reconstruct.

Thanks to the seminal
work of Landau, Pollak and Slepian, the optimal orthogonal system for representing
almost time and band limited functions is known.
The system in questions consists of the so called prolate spheroidal wave functions, $\psi_k^T$,
and has many valuable properties (see \cite{prolate1,prolate2,prolate3,prolate4,Slepian4}). Among the most striking properties they have
is that, if a function is almost time limited to $[-T,T]$ and almost band limited to $[-\Omega,\Omega]$
then it is well approximated by its projection on the first $4\Omega T$ terms of the basis:
\rea{with these assumptions, it is $4\Omega T$ and not $8\Omega T$}
\begin{equation}
\label{eq:prolate}
f\simeq\sum_{0\leq k<4\Omega T}\scal{f,\psi_k^T}\psi_k^T.
\end{equation}
For more details, see \cite{prolate2}. 
This is a remarkable fact as this is exactly the heuristics given by Shannon's sampling formula
(note that to make this heuristics clearer, the functions are usually almost time-limited to $[-T/2,T/2]$
and this result is then known as the $2\Omega T$-Theorem, see \cite{prolate2}).

However, there is a major difficulty with prolate spheroidal wave functions that has attracted  a lot of interest recently,
namely the difficulty to compute them as there is no inductive nor closed form formula (see e.g. \cite{BKnote,BK1,Boyd1,Li,Xiao}). One approach
is to explicitly compute the coefficients of the prolate spheroidal wave functions in terms of a basis of orthogonal
polynomials like the Legendre polynomials or  the Hermite functions basis. The question that then arises
is that of directly approximating almost time and band limited functions by the (truncation of) their expansion
in the Hermite, Legendre and Chebyshev  bases. This is the question we address here.

An other motivation for this work comes from the work of the first author \cite{JP} on uncertainty principles for orthonormal bases.
There, it is shown that an orthonormal basis $(e_k)$ of $L^2(\R)$
can not have uniform time-frequency localization. Several ways of measuring
localization were considered, and for most of them, the Hermite functions provided the optimal behavior.
However, in one case, the proof relied on \eqref{eq:prolate}: this shows that the set of functions
that are $\eps_T$-time limited to $[-T,T]$ and $\eps_\Omega$-band limited to $[-\Omega,\Omega]$
is almost of dimension $4\Omega T$. In particular, this set can not contain more than
a fixed number of elements of an orthonormal sequence. As this proof shows, the optimal basis
here consists of prolate spheroidal wave functions. As the Hermite basis is optimal
for many uncertainty principles, it is thus natural to ask how far it is from optimal in this case.

Let us now be more precise and describe the main results of the paper.  In Section 2, we first give a brief description of the 
asymptotic approximation of the Hermite functions in terms of the sine and cosine functions. Then, we use the asymptotic behaviour of the Hermite function and give an error analysis of the uniform approximation of the Hermite function projection kernel ${\displaystyle k_n(x,y)=
\sum_{k=0}^n h_k(x)h_k(y)}$ by an appropriate Sinc kernel. Here, $h_k$ denotes the $k-$th $L^2$-normalized Hermite function.  Then, based on the previous 
asymptotic approximation of the Hermite kernel,  we give the quality of almost time- and band-limited functions by Hermite functions. In Section 3, we
use the explicit formula for the finite Fourier transform of the Legendre polynomials in terms of the Bessel function and give the convergence rate 
of the Legendre series expansion of a $c-$band-limited function. Then, we extend this result to the case of almost time- and band-limited function. 
In Section 4, we show the results obtained for the Legendre polynomials to the case of Chebyshev polynomials. Section 5 is divided into two parts. 
In the first part, we first give an application of the results of Section 3 related to the convergence rate of the Legendre series expansion of the prolate spheroidal wave functions (PSWFs). Note that for a given bandwidth $c>0,$ and an integer $n\geq 0,$  the $n-$th PSWF, denoted by $\psi_{n,c}$ is a $c-$band-limited function, given as the $n-$th eigenfunction of a compact integral operator $Q_c,$ defined on $L^2([-1,1])$ with the sinc kernel ${\displaystyle K_c(x,y)=\frac{\sin c(x-y)}{\pi (x-y)}.}$ In the second part of Section 5, we give various numerical examples  that illustrate the 
different results of this work.

\medskip

\section{Approximation of almost band limited functions by Hermite functions basis.}\label{sec:wkb}

In this section, we study the quality of approximation of band limited and almost band limited functions by the Hermite and scaled Hermite functions.
For this purpose, we first need to review the asymptotic uniform approximation of the Hermite functions by the sine and cosine functions. 
This is the subject of the following paragraph.

\subsection{Approximating Hermite functions with the WKB method.}

Let $H_n$ be the $n$-th Hermite polynomial, that is
$$
H_n(x)=e^{x^2}\frac{\mathrm{d}^n}{\mathrm{d}x^n}e^{-x^2}.
$$
Define the Hermite functions as
$$
h_n(x)=\alpha_nH_n(x)e^{-x^2/2}\quad\mbox{where }\alpha_n=\frac{1}{\pi^{1/4}\sqrt{2^nn!}}.
$$
As is well known:
\begin{enumerate}
\renewcommand{\theenumi}{\roman{enumi}}
\item $(h_n)_{n\geq 0}$ is an orthonormal basis of $L^2(\R)$.
\item $h_n$ is even if $n$ is even and odd if $n$ is odd, in particular
$h_{2p}^{\prime}(0)=0$ and $h_{2p+1}(0)=0$.
Further
$$
h_{2p}(0)=\frac{(-1)^p}{\pi^{1/4}}\sqrt{\frac{(2p-1)!!}{(2p)!!}}
\quad\mbox{and}\quad
h_{2p+1}^{\prime}(0)=\frac{(-1)^p\sqrt{4p+2}}{\pi^{1/4}}\sqrt{\frac{(2p-1)!!}{(2p)!!}}.
$$

\item $h_n$ satisfies the differential equation $h_n^{\prime\prime}(x)+(2n+1-x^2)h_n(x)=0$.
\end{enumerate}

We will now follow the WKB method to obtain an approximation of $h_n$.
In order to simplify notation, we will fix $n$ and drop all supscripts during the computation.
Let $h=h_n$, $\lambda=\sqrt{2n+1}$, and define for $|x|<\lambda$
$$
p(x)=\sqrt{\lambda^2-x^2},\quad\ffi(x)=\int_0^xp(t)\,\mbox{d}t
\quad\mbox{and}\quad
\psi_\pm(x)=\frac{1}{\sqrt{p(x)}}\exp\pm i\ffi(x).
$$

Note that $\psi_\pm$ have been chosen to have
$$
\psi_+(x)\psi_-^\prime(x)-\psi_-(x)\psi_+^\prime(x)=-2i
$$
and
$$
y''+(p^2-q)y=0\qquad\mbox{where }q=\frac{1}{2}\left(\frac{p'}{p}\right)'-\frac{1}{4}\left(\frac{p'}{p}\right)^2
=-\dst\frac{2\lambda^2+3x^2}{4p(x)^4}.
$$
Note that $h''(x)+p(x)h(x)=0$ so that
$$
(h'\psi_\pm-\psi_\pm^{\prime}h)'=h''\psi_\pm-\psi_\pm^{\prime\prime}h=-qh\psi_\pm.
$$
Let us now define
$$
Q_\pm(x)=\int_0^xq(t)h(t)\psi_\pm(t)\,\mbox{d}t.
$$
Integrating the previous differential equation  between $0$ and $x$, we obtain the system
$$
\left\{\begin{matrix}
h'(x)\psi_+(x)&-&h(x)\psi_+^{\prime}(x)&=&h'(0)\psi_+(0)-h(0)\psi_+^{\prime}(0)&-&Q_+(x)\\
h'(x)\psi_-(x)&-&h(x)\psi_-^{\prime}(x)&=&h'(0)\psi_-(0)-h(0)\psi_-^{\prime}(0)&-&Q_-(x)\\
\end{matrix}\right..
$$
It remains to solve this system for $h$ to obtain the principal term of $h$:

\begin{theorem}
\label{th:approxherm}
Let $n\geq 0$, $\lambda=\sqrt{2n+1}$. Then, for $|x|\leq\lambda$,
\begin{equation}
h_n(x)=\sqrt{\lambda}h_n(0)\frac{\cos\ffi_n(x)}{(\lambda^2-x^2)^{1/4}}+\frac{h_n^{\prime}(0)}{\sqrt{\lambda}}
\frac{\sin\ffi_n(x)}{(\lambda^2-x^2)^{1/4}}+E_n(x)
\label{eq:approxherm}
\end{equation}
where
\begin{equation}
\label{eq:approxhermest}
\ffi_n(x)=\int_0^x\sqrt{\lambda^2-t^2}\,\mbox{d}t
\quad\mbox{and}\quad
|E_n(x)|\leq\frac{5}{4}\left(\frac{\lambda}{\lambda^2-x^2}\right)^{5/2}.
\end{equation}
Further, if $|x|,|y|\leq T\leq\frac{\lambda}{2}$,
$$
\ffi_n(x)=\sqrt{2n+1}x-e_n(x),
$$
where
\begin{equation}
\label{eq:esten}
|e_n(x)|\leq \frac{T^3}{3\lambda}
\quad\mbox{and}\quad
|e_n(x)-e_n(y)|\leq \frac{T^2}{\lambda}|x-y|,
\end{equation}
while
\begin{equation}
\label{eq:estEn}
|E_n(x)|\leq\frac{2}{\lambda^3}
\quad\mbox{and}\quad |E_n(x)-E_n(y)|\leq\frac{7}{\lambda^{5/2}}|x-y|.
\end{equation}
\end{theorem}

\begin{remark}
One may explicitly compute $\ffi$:
$$
\ffi_n(x)=\frac{2n+1}{2}\arcsin\frac{x}{\sqrt{2n+1}}+\frac{x}{2}\sqrt{2n+1-x^2}.
$$
Also, $\ffi_n$ has a geometric interpretation: it this the area of the intersection of a disc of radius $\sqrt{2n+1}$
centered at $0$ with the strip $[0,x]\times\R^+$. In particular, when $x\to\sqrt{2n+1}$, $\ffi_n(x)\sim \frac{\pi}{4}(2n+1)$.
\end{remark}

The result is not entirely new ({\it e.g.} \cite{BKH,Do,kochtataru,larsson,Sa}), except for the Lipschitz bounds of $E$. Therefore we will only sketch the proof of this theorem in Appendix A.

Using standard asymptotic of $h_{2p}(0)$ and of $h_{2p+1}^\prime(0)$ and the fact that $\sqrt{\lambda^2-x^2}\simeq\lambda$
when $\lambda\to\infty$, one may further simplify this result
to the following:

\begin{corollary}
\label{cor:simpasympt}
Let $T\geq 2$ and let $n\geq 2T^2$. Then, for $|x|\leq T$, we obtain that

-- if $n$ is even, $n=2p$
\begin{equation}
\label{eq:approxhermpair}
h_{2p}(x)=\frac{(-1)^p}{\sqrt{\pi}p^{1/4}}\cos\ffi_{2p}(x)+\tilde E_{2p}(x);
\end{equation}

-- if $n$ is odd, $n=2p+1$
\begin{equation}
\label{eq:approxhermimpair}
h_{2p+1}(x)=
\frac{(-1)^p}{\sqrt{\pi}p^{1/4}}\sin\ffi_{2p+1}(x)+\tilde E_{2p+1}(x),
\end{equation}
where, for $|x|,|y|\leq T$,
\begin{equation}
\label{eq:esttildeE}
|\tilde E_n(x)|\leq\frac{3T^2}{(2n+1)^{5/4}}
\quad\mbox{and}\quad
|\tilde E_n(x)-\tilde E_n(y)|\leq 8\frac{T^2}{(2n+1)^{3/4}}|x-y|
\end{equation}
\end{corollary}

To conclude, we will gather some facts about $\ffi_n$ that all follow from easy calculus.

\begin{lemma}
\label{lem:simpasympt}
 If $|x|,|y|\leq T\leq\frac{1}{2}\sqrt{2n+1}$, then
\begin{equation}
\label{eq:lipffi0}
|\ffi_{n+1}(x)-\ffi_n(x)|\leq\frac{3T}{\sqrt{2n+1}},
\end{equation}
\begin{equation}
\label{eq:lipffi}
|\ffi_{n+1}(x)-\ffi_{n+1}(y)-\ffi_n(x)+\ffi_n(y)|\leq\frac{3}{\sqrt{2n+1}}|x-y|,
\end{equation}
\begin{equation}
\label{eq:lipffi2}
|\ffi_{n+1}(x)-\ffi_{n}(x)+\ffi_{n+1}(y)-\ffi_{n}(y)|\leq\frac{5T}{\sqrt{2n+1}},
\end{equation}
\begin{equation}
\label{eq:ffi}
\ffi_{n+1}(x)+\ffi_n(x)-\ffi_{n+1}(y)-\ffi_n(y)=(\sqrt{2n+1}+\sqrt{2n+3})(x-y)+\eps_n(x,y),
\end{equation}
with $\dst|\eps_n(x,y)|\leq\frac{T^2}{\sqrt{2n+1}}|x-y|$ and
\begin{equation}
\label{eq:lipffixxx}
|\ffi_{n}(x)-\ffi_{n}(y)|\leq\frac{5}{4}\sqrt{2n+1}|x-y|.
\end{equation}
\end{lemma}

\subsection{The kernel of the projection onto the Hermite functions}\label{sec:kernel}

As $(h_n)_{n\geq 0}$ forms an orthonormal basis of $L^2(\R)$, every $f\in L^2(\R)$ can be written as
$$
f(x)=\lim_{n\to+\infty}\sum_{k=0}^n \scal{f,h_k}h_k(x),
$$
where the limit is in the $L^2(\R)$ sense. Further,
for $n$ an integer, let $K_nf$ be the orthogonal projection of $f$
on the span of $h_0,\ldots,h_n$. Then
$$
K_nf(x)=\sum_{k=0}^n \scal{f,h_k}h_k(x)
=
=\int_{\R} k_n(x,y)f(y)\,\mbox{d}y,
$$
with the kernel
$\dst k_n(x,y)=\sum_{k=0}^nh_k(x)h_k(y)$.
According to the Christoffel-Darboux Formula,
$$
k_n(x,y)=\sqrt{\frac{n+1}{2}}\frac{h_{n+1}(x)h_n(y)-h_{n+1}(y)h_n(x)}{x-y}.
$$
We will now use Corollary \ref{cor:simpasympt} to approximate this kernel:

\begin{theorem}
\label{th:estkn}
Let $T\geq 2$, $n\geq 2T^2$ and $N=\frac{\sqrt{2n+1}+\sqrt{2n+3}}{2}$.
Then, for  $|x|,|y|\leq T$,
$$
k_n(x,y)=\frac{1}{\pi}\frac{\sin N(x-y)}{x-y}+R_n(x,y),
$$
with $|R_n(x,y)|\leq\dst\frac{17T^2}{\sqrt{2n+1}}$.
\end{theorem}

\begin{remark} The same estimate holds for $T=1$ provided $n\geq 6$. Moreover, we should mention that in practice, the actual
approximation error of the kernel is much smaller than the theoretical error $R_n.$ See example 1 in the numerical results section
that illustrate this fact.
\end{remark}

Again, the only improvement over known results \cite{Sa,Us}
is in the estimate of $R_n$. We will therefore only sketch the proof in Appendix B.

\subsection{Approximating almost time and band limited functions by Hermite functions}\label{sec:mainth}

We can now prove the following theorem.

\begin{theorem}\label{th:projhermite}
Let $\Omega_0,T_0\geq2$ and $\eps_T,\eps_\Omega>0$. Assume that
$$
\int_{|t|>T_0}|f(t)|^2\,\mathrm{d}t\leq \eps_T^2\norm{f}_{L^2(\R)}^2
\quad\mbox{and}\quad
\int_{|\omega|>\Omega_0}|\widehat{f}(\omega)|^2\,\mathrm{d}\omega\leq \eps_\Omega^2\norm{f}_{L^2(\R)}^2.
$$

Assume that $n\geq \max(2T^2,2\Omega^2)$. Then, for $T\geq T_0$,
\begin{equation}
\label{eq:approxhermloc}
\norm{f-K_nf}_{L^2([-T,T])}
\leq\left(2\eps_T+\eps_\Omega+\frac{34T^3}{\sqrt{2n+1}}\right)\norm{f}_{L^2(\R)}
\end{equation}
\end{theorem}

\begin{proof}
We will introduce several projections. For $T,\Omega>0$, let
$$
P_Tf=\mathbf{1}_{[-T,T]}f\quad\mbox{and}\quad Q_\Omega f(x)=\ff^{-1}\bigl[\mathbf{1}_{[-\Omega,\Omega]}\widehat{f}](x)=\frac{1}{\pi}\int_{\R}\frac{\sin\Omega(x-y)}{x-y}f(y)\,\mbox{d}y.
$$
The hypothesis on $f$ is that $\norm{f-P_Tf}_{L^2(\R)}\leq \eps_T\norm{f}_{L^2(\R)}$ for $T\geq T_0$
and $\norm{f-Q_\Omega f}_{L^2(\R)}\leq\eps_\Omega\norm{f}_{L^2(\R)}$ for $\Omega\geq\Omega_0$.
Let us also define the integral operator
$$
\mathcal{R}_n^T f(x)=\int_{[-T,T]}R_n(x,y)f(y)\,\mbox{d}y,
$$
where $R_n(x,y)$ are defined in Theorem \ref{th:estkn}.
Notice that $k_n(x,y)=k_n(y,x)$ so that $R_n(x,y)=R_n(y,x)$.

\medskip

It is enough to prove \eqref{eq:approxhermloc} for $T=T_0$.
We may then reformulate Theorem \ref{th:estkn} as following:
$$
P_TK_nP_Tf=P_TQ_NP_Tf+P_T\mathcal{R}_n^T P_Tf,
$$
where $N=\frac{\sqrt{2n+1}+\sqrt{2n+3}}{2}$. Note that $N\geq\Omega_0$.
By using \eqref{th:estkn}, it is easy to see that
\begin{eqnarray}
\norm{P_T \mathcal R_n^T P_Tf}_{L^2(\R)}&\leq& \norm{P_T \mathcal R_n^T P_T}_{L^2(\R)\to L^2(\R)}\norm{f}_{L^2(\R)}\nonumber\\
&\leq&
\| P_T\mathcal R_n^TP_T\|_{HS} \norm{f}_{L^2(\R)}\nonumber\\
&\leq&\frac{34T^3}{\sqrt{2n+1}}\norm{f}_{L^2(\R)}.\label{norm_Rn}
\end{eqnarray}
Here we use the well known fact that the Hilbert-Schmidt norm of an integral operator is the $L^2$ norm of its kernel.
%

Now, using the fact that projections are contractive and $N\geq\Omega_0$, we have
\begin{eqnarray*}
\norm{f-K_nf}_{L^2([-T,T])}&=&\norm{P_Tf-P_TK_nf}_{L^2(\R)}\\
&\leq& \norm{P_Tf-P_TK_nP_Tf}_{L^2(\R)}+\norm{P_TK_n(f-P_Tf)}_{L^2(\R)}\\
&\leq&\norm{P_Tf-P_TQ_NP_Tf+P_T\mathcal{R}_n^T P_Tf}_{L^2(\R)}+\norm{f-P_Tf}_{L^2(\R)}\\
&\leq&\norm{P_Tf-P_TQ_NP_Tf}_{L^2(\R)}+\norm{P_T\mathcal{R}_n^T P_Tf}_{L^2(\R)}+\norm{f-P_Tf}_{L^2(\R)}.
\end{eqnarray*}
Now, write $P_TQ_NP_Tf=P_TQ_Nf+P_TQ_N(f-P_Tf)$, then
\begin{eqnarray*}
\norm{P_Tf-P_TQ_NP_Tf}_{L^2(\R)}&\leq& \norm{P_Tf-P_TQ_Nf}_{L^2(\R)}+\norm{P_TQ_N(f-P_Tf)}_{L^2(\R)}\\
&\leq&\norm{f-Q_Nf}_{L^2(\R)}+\norm{f-P_Tf}_{L^2(\R)}.
\end{eqnarray*}

Therefore,
\begin{eqnarray*}
\norm{f-K_nf}_{L^2([-T,T])}
&\leq&\norm{f-Q_Nf}_{L^2(\R)}+\frac{34T^3}{\sqrt{2n+1}}\norm{f}_{L^2(\R)}+2\norm{f-P_Tf}_{L^2(\R)}\\
&\leq&\left(\eps_\Omega+\frac{34T^3}{\sqrt{2n+1}}+2\eps_T\right)\norm{f}_{L^2(\R)},
\end{eqnarray*}
since $N\geq\Omega_0$.
\end{proof}

\rep{Extended remark}
\begin{remark}
The error estimate given by \eqref{eq:approxhermloc} is not practical due to the  low decay rate
of the bound of $\|\mathcal R_n^T\|$ given by ${\displaystyle \frac{34 T^3}{\sqrt{2n+1}}}.$
By replacing this with a non explicit but a more realistic error estimate $\|\mathcal R_n^T\|_{HS},$
one gets the following  error estimate  which is more practical for numerical purposes,
 \begin{equation}\label{eq2:approxhermloc}
  \norm{f-K_nf}_{L^2([-T,T])}
 \leq\left(\eps_\Omega+\| \mathcal R_n^T\|_{HS}+2\eps_T\right)\norm{f}_{L^2(\R)}.
 \end{equation}
\end{remark}

\subsection{Approximating almost time and band limited functions by scaled Hermite functions}\label{sec:scale}
\rep{Inverted the way we scale here to be coherent with the next section}

For $\alpha>0$ and $f\in L^2(\R)$ we define the scaling operator $\delta_\alpha f(x)=\alpha^{-1/2}f(x/\alpha)$.
Recall that $\norm{\delta_\alpha f}_{L^2(\R)}=\norm{f}_{L^2(\R)}$ while
$$
\norm{\delta_\alpha f}_{L^2([-A,A])}=\norm{f}_{L^2([-A/\alpha ,A/\alpha A])},\
\norm{\delta_\alpha f}_{L^2(\R\setminus[-A,A])}=\norm{f}_{L^2(\R\setminus[-A/\alpha,A/\alpha])}
$$
and $\ff[\delta_\alpha f]=\delta_{1/\alpha}\ff[f]$.
In particular, if $f$ is $\eps_T$-almost time limited to $[-T,T]$ (resp. $\eps_\Omega$-almost band limited to $[-\Omega,\Omega]$) then
$\delta_\alpha f$ is $\eps_T$-almost time limited to $[-T/\alpha,T/\alpha]$
(resp. $\eps_\Omega$-almost band limited to $[-\alpha\Omega,\alpha\Omega]$).

Next, define the scaled Hermite basis $h_k^\alpha=\delta_{\alpha}h_k$
which is also an orthonormal basis of $L^2(\R)$ and define the corresponding orthogonal projections:
for $f\in L^2(\R)$,
\begin{equation}
\label{scaled_approx}
K^\alpha_nf=\sum_{k=0}^n\scal{f,h_k^\alpha}h_k^\alpha.
\end{equation}

\begin{proposition}\label{th:scaled}
Let $\alpha>0$, $T\geq2$ and $c\geq 2/\alpha$.
Assume that and
$$
\int_{|t|>T}|f(t)|^2\,\mathrm{d}t\leq \eps_T^2\norm{f}_{L^2(\R)}^2
\quad\mbox{and}\quad
\int_{|\omega|>c/\alpha}|\widehat{f}(\omega)|^2\,\mathrm{d}\omega\leq \eps_{c/\alpha}^2\norm{f}_{L^2(\R)}^2.
$$
Then, for $n\geq \max(2(T/\alpha)^2 ,2c^2)$, we have
\begin{equation}
\label{approx_scaled}
\norm{f-K_n^\alpha f}_{L^2([-T,T])}\leq\left(\eps_T+\eps_{c/\alpha}+\frac{34(T/\alpha)^3}{\sqrt{2n+1}}\right)\norm{f}_{L^2(\R)}.
\end{equation}
\end{proposition}

\begin{remark}
The scaling with $\alpha>1$ has as effect to decrease the dependence on $T$ at the price of increasing
the dependence on good frequency concentration, while taking $\alpha<1$ the gain and loss are reversed.
In practice, the above dependence on $T$ is very pessimistic and $\alpha>1$ is a better choice. The most natural choice is $\alpha=T$
and $c=T\Omega$ where $\Omega$ is such that $f$ is $\eps_\Omega$-almost band limited to $[-\Omega,\Omega]$.
\end{remark}

\begin{proof} For $f\in L^2(\R)$, since $K_n^\alpha$ is contractive, we have
\begin{eqnarray*}
\norm{f-K_n^\alpha f}_{L^2([-T,T])}&\leq& \norm{f-K_n^\alpha P_T f}_{L^2([-T,T])}+\norm{K_n^\alpha(f- P_T f)}_{L^2([-T,T])}\\
&\leq& \norm{f-K_n^\alpha P_T f}_{L^2([-T,T])}+\norm{f- P_T f}_{L^2([-T,T])}\\
&\leq&\norm{f-K_n^\alpha P_T f}_{L^2([-T,T])}+\eps_T\norm{f}_{L^2(\R)}.
\end{eqnarray*}
Moreover,
\begin{eqnarray*}
K_n^\alpha P_T f(x)&=&\sum_{k=0}^n\scal{P_T f,h_k^\alpha}h_k^\alpha(x)=\int_{-T}^T
f(y)\frac{1}{\alpha}\sum_{k=0}^nh_k(x/\alpha)h_k(y/\alpha)\,\mbox{d}y\\
&=&\int_{-T/\alpha}^{T/\alpha}f(\alpha t)\sum_{k=0}^nh_k(x/\alpha)h_k(t)\,\mbox{d}t.
\end{eqnarray*}
From this, one easily deduces that $\norm{f-K_n^\alpha P_T f}_{L^2([-T,T])}=\norm{f_\alpha-K_nf_\alpha}_{L^2([-\alpha T,\alpha T])}$
where $f_\alpha=\delta_{1/\alpha}\bigl[\mathbf{1}_{[-T,T]}f\bigr]$.
Note that $f_\alpha$ is $0$-almost time limited to $[-T/\alpha,T/\alpha]$. Next, writing
$$
\widehat{f_\alpha}=\delta_{\alpha}\ff[\mathbf{1}_{[-T,T]}f]=\delta_{\alpha}\ff[f]
-\delta_{\alpha}\ff[\mathbf{1}_{\R\setminus [-T,T]}f]
$$
and, noting that
$$
\norm{\delta_{\alpha}\ff[f]}_{L^2(\R\setminus[-c,c])}=\norm{\ff[f]}_{L^2(\R\setminus[-c/\alpha,c/\alpha])}
\leq\eps_{c/\alpha}\norm{f}_{L^2(\R)}
$$
while
\begin{eqnarray*}
\norm{\delta_{\alpha}\ff[\mathbf{1}_{\R\setminus [-T,T]}f]}_{L^2(\R\setminus[-\Omega,\Omega])}
&\leq&\norm{\delta_{\alpha}\ff[\mathbf{1}_{\R\setminus [-T,T]}f]}_{L^2(\R)}\\
&=&\norm{\mathbf{1}_{\R\setminus [-T,T]}f}_{L^2(\R)}\leq\eps_T\norm{f}_{L^2(\R)},
\end{eqnarray*}
we get
$$
\norm{\widehat{f_\alpha}}_{L^2(\R\setminus[-c,c])}
\leq \eps_{c/\alpha}\norm{f}_{L^2(\R)}+\eps_T\norm{f}_{L^2(\R)}.
$$
It remains to apply Theorem \ref{th:projhermite} to complete the proof.
\end{proof}

\section{Approximation of almost band limited functions in the basis of Legendre polynomials}

In agreement with standard practice, we will denote by $P_k$ the classical Legendre
polynomials, defined by the three-term recursion
$$
P_{k+1}(x) =\frac{2k + 1}{k + 1}x P_k(x) - \frac{k}{k + 1}P_{k-1}(x),
$$
with the initial conditions
$$
P_0(x) = 1, P_1(x) = x.
$$
These polynomials are orthogonal in $L^2([-1,1])$ and are normalized so that
$$
P_k(1)=1\quad\mbox{and}\quad\int_{-1}^1P_k(x)^2\,\mbox{d}x=\frac{1}{k+1/2}.
$$
We will denote by $\tilde P_k$ the normalized Legendre polynomial
$\tilde P_k=\sqrt{k+1/2}P_k$ and the $\tilde P_k$'s then form an orthonormal basis of $L^2([-1,1])$.

\medskip

In the sequel, for $c>0$, let $B_c$ denote the Paley-Wiener space of $c$-bandlimited functions, given by
$$
\bb_c=\{ f\in L^2(\mathbb R);\,\, \mbox{Supp } \widehat{f} \subseteq [-c,c]\}.
$$

\begin{lemma}
\label{lem:decay_Bc}
Let $c>0,$ then for any $f\in \bb_c$, and any $k\geq 0$
\begin{equation}\label{eq:2.1}
|\scal{f, P_k}_{L^2(-1,1)}|\leq \frac{2}{\sqrt{2k+1}} \sqrt{\frac{e}{\pi c}} \left(\frac{ec}{2k+3}\right)^{k+1} \| f\|_{L^2(\mathbb R)}.
\end{equation}
\end{lemma}

\begin{proof}
We start from the following identity relating Bessel functions of the first type to the finite Fourier transform
of the Legendre polynomials,  {\it see} \cite{Andrews}: for every $x\in\R$
\begin{equation}
\label{eq:2.2.9}
\int_{-1}^1 e^{i x y}  P_k(y)\,\mbox{d}y =2i^kj_k(x)
\end{equation}
where $j_k$ is the spherical Bessel function defined by
$j_k(x)=\dst(-x)^k\left(\frac{1}{x}\frac{\mbox{d}}{\mbox{d}x}\right)^k\frac{\sin x}{x}$.
Note that $j_k$ has same parity as $n$ and recall that, for $x\geq0$,
$j_k(x)=\dst\sqrt{\frac{\pi}{2x}}J_{k+1/2}(x)$ where $J_\alpha$ is the Bessel function
of the first kind. In particular, we have the well known bound for $x\in\R$
\begin{equation}
\label{eq:bessel}
|J_{\alpha}(x)|\leq \frac{|x|^{\alpha}}{2^{\alpha}\Gamma(\alpha+1)}
\leq \frac{e^{\alpha+1}}{\sqrt{2\pi}2^{\alpha}(\alpha+1)^{\alpha+1/2}}|x|^{\alpha}
\end{equation}
since $\Gamma(x)\geq\sqrt{2\pi}x^{x-1/2}e^{-x}$. From this we deduce that
\begin{equation}
\label{eq:boundjn}
|j_k(x)|\leq \frac{e^{k+3/2}}{\sqrt{2}(2k+3)^{k+1}}|x|^k.
\end{equation}

Now, since  $f\in \bb_c$, the Fourier inversion theorem implies that, for $x\in \R$, we have
\begin{equation}
\label{eq:1.3}
f(x)= \frac{1}{\sqrt{2\pi}}\int_{-c}^c \widehat{f}(\xi) e^{i\, x\, \xi}\, \mbox{d}\xi
= \frac{c}{\sqrt{2\pi}}\int_{-1}^1 \widehat{f}(c\eta) e^{i\, c\, x\, \eta}\, \mbox{d}\eta.
\end{equation}
Combining \eqref{eq:2.2.9} and \eqref{eq:1.3},  one gets
\begin{eqnarray*}
\scal{f,P_k}_{L^2([-1,1])}&=&\int_{-1}^1 f(x){P_k(x)}\,\mbox{d}x
=\frac{c}{\sqrt{2\pi}} \int_{-1}^1  \widehat{f}(c\eta)\left(\int_{-1}^1e^{-i c x\eta} {P_k(x)}\, \mbox{d}x\right)
\, \mbox{d}\eta\\
&=& i^k c\sqrt{\frac{2}{\pi}}\int_{-1}^1 j_{k}(c\eta) \widehat{f}(c\eta)\, d\eta.
\end{eqnarray*}
Using \eqref{eq:boundjn} together with  Cauchy-Schwarz and a change of variable, one gets
\begin{eqnarray*}
|\scal{f,P_k}_{L^2([-1,1])}|&\leq & c^{k+1}\frac{e^{k+3/2}}{\sqrt{\pi}(2k+3)^{k+1}}
\int_{-1}^1 |\eta|^k |\widehat{f}(c\eta)|\, d\eta \\
&\leq & c^{k+1}\frac{e^{k+3/2}}{(2k+3)^{k+1}}\sqrt{\frac{2}{2k+1}}\sqrt{\frac{2}{\pi}}
\left(\frac{1}{c}\int_{-c}^c  |\widehat{f}(\eta)|^2\, d\eta\right)^{1/2}. \\
\end{eqnarray*}
Finally, Parseval's identity implies \eqref{eq:2.1}.
\end{proof}

Let us now introduce the following orthogonal projections on $L^2(\R)$:
$$
Pf=\mathbf{1}_{(-1,1)}f,\quad
Q_c f=\ff^{-1}[\mathbf{1}_{(-c,c)}\ff f]
\quad\mbox{and}\quad
\ll_Nf=\sum_{k=0}^{N}\scal{Pf,\tilde P_k}\tilde P_k\mathbf{1}_{(-1,1)}.
$$
Note that $\ll_N$ is the orthogonal projection onto the subspace of $L^2(\R)$ consisting
of functions of the $P(x)\mathbf{1}_{(-1,1)}$ with $P$ a polynomial of degree $\leq N$.

\begin{theorem}
Let $c>0,$ then for any $f\in \bb_c$, and any $N\geq \frac{ec}{2},$ we have
\begin{equation}
\label{eq:approxbcLinfini}
\norm{f-\ll_Nf}_{L^\infty(-1,1)}\leq \sqrt{\frac{c}{2N+5}}\left(\frac{ec}{2N+5}\right)^{N}\| f\|_{L^2(\mathbb R)}.
\end{equation}
and
\begin{equation}
\label{eq:approxbcL2}
\norm{f-\ll_Nf}_{L^2(-1,1)}\leq \sqrt{c}\left(\frac{ec}{2N+5}\right)^{N+1}\| f\|_{L^2(\mathbb R)}.
\end{equation}
\end{theorem}

\begin{proof} Note that, for $x\in(-1,1)$,
$$
f(x)-\ll_Nf(x)=\sum_{k=N+1}^{+\infty}\langle f,\tilde P_k\rangle\tilde P_k(x).
$$
But $\dst\max\limits_{x\in(-1,1)}|\tilde P_k(x)|=|\tilde P_k(1)|=\sqrt{k+1/2}$,
so that Lemma \ref{lem:decay_Bc} implies
\begin{eqnarray*}
\norm{f-\ll_Nf}_{L^\infty(-1,1)}&\leq&\sum_{k=N+1}^{+\infty}(k+1/2)|\langle f,P_k\rangle|\\
&\leq&\sqrt{\frac{e}{\pi c}}\sum_{k=N+1}^{+\infty}\sqrt{2k+1} \left(\frac{ec}{2k+3}\right)^{k+1} \| f\|_{L^2(\mathbb R)}\\
&\leq&\frac{e}{\sqrt{2N+5}}\sqrt{\frac{e c}{2\pi}}\sum_{k=N+1}^{+\infty} \left(\frac{ec}{2N+5}\right)^{k}\| f\|_{L^2(\mathbb R)}\\
&\leq&\sqrt{\frac{c}{2N+5}}\left(\frac{ec}{2N+5}\right)^{N}\| f\|_{L^2(\mathbb R)}.
\end{eqnarray*}
If $N\geq ec/2$, we then deduce \eqref{eq:approxbcLinfini}.

The proof of the $L^2$-bound is essentially the same:
\begin{eqnarray*}
\norm{f-\ll_Nf}_{L^2(-1,1)}^2&\leq&\sum_{k=N+1}^{+\infty}(k+1/2)|\langle f,P_k\rangle|^2\\
&\leq&\frac{e}{2\pi c}\sum_{k=N+1}^{+\infty} \left(\frac{ec}{2k+3}\right)^{2k+2} \| f\|_{L^2(\mathbb R)}^2\\
&\leq&\frac{e^2}{2\pi}\sum_{k=N+1}^{+\infty}\left(\frac{ec}{2N+5}\right)^{2k+2}\ \| f\|_{L^2(\mathbb R)}^2.
\end{eqnarray*}
From this \eqref{eq:approxbcL2} easily follows when $N\geq ec/2.$
\end{proof}

From this theorem, we simply get the following corollary:

\begin{theorem}
\label{th:leg}
Let $c>0$ and assume that $f$ is $\eps_T$-concentrated to $(-1,1)$ and $\eps_\Omega$-concentrated to $(-c,c)$.
Then, if $N\geq ec/2$,
\begin{equation}
\label{eq:mainth1}
\norm{f-\ll_Nf}_{L^2(-1,1)}\leq\left(2\eps_\Omega+\sqrt{c}\left(\frac{ec}{2N+5}\right)^{N+1}\right)\norm{f}_{L^2(\R)}
\end{equation}
and
\begin{equation}
\label{eq:mainth2}
\norm{f-\ll_Nf}_{L^2(\R)}\leq\left(\eps_T+2\eps_\Omega+\sqrt{c}\left(\frac{ec}{2N+5}\right)^{N+1}\right)\norm{f}_{L^2(\R)}
\end{equation}
\end{theorem}

\begin{proof}
First
\begin{eqnarray*}
\norm{f-\ll_Nf}_{L^2(-1,1)}&\leq&\norm{f-Q_cf}_{L^2(-1,1)}
+\norm{Q_cf-\ll_NQ_cf}_{L^2(-1,1)}+\norm{\ll_N(Q_cf-f)}_{L^2(-1,1)}\\
&\leq& 2\norm{f-Q_cf}_{L^2(\R)}+\norm{Q_cf-\ll_NQ_cf}_{L^2(-1,1)}.
\end{eqnarray*}
But $\norm{f-Q_cf}_{L^2(\R)}\leq \eps_\Omega\norm{f}_{L^2(\R)}$ and $Q_cf\in\bb_c$
with $\norm{Q_cf}_{L^2(\R)}\leq\norm{f}_{L^2(\R)}$.

It remains to notice that
$$
\norm{f-\ll_Nf}_{L^2(\R)}\leq
\norm{f-P_T f}_{L^2(\R)}+\norm{f-\ll_Nf}_{L^2(-1,1)}
$$
so that \eqref{eq:mainth2} follows.
\end{proof}

\section{Approximation of almost band limited functions in the basis of Chebyshev polynomials}

In this paragraph, we show that the basis of the Chebyshev polynomials is also well adapted for the
approximation of almost band limited functions. This is essentially done by showing that  the weighted finite Fourier transform
of the Chebyshev polynomial is given by a formula similar to \eqref{eq:2.2.9}. We first recall that
the classical Chebyshev polynomials $T_k$ are  defined by the three-term recursion
$$
T_{k+1}(x) =2 x T_k(x) - T_{k-1}(x),
$$
with the initial conditions
$$
T_0(x) = 1,\quad T_1(x) = x.
$$
These polynomials are orthogonal in $L^2([-1,1],\, \mbox{d}\mu)$ where $\mbox{d}\mu(x)= \frac{1}{\sqrt{1-x^2}}\,\mbox{d}x$  and are normalized so that
\begin{equation}\label{normalization}
T_k(1)=1\quad\mbox{and}\quad\int_{-1}^1 T_n(x)^2\,\mbox{d}\mu(x)=c_k \frac{\pi}{2} \quad\mbox{with}\quad c_k=
\begin{cases} 2&\mbox{if } k=0\\  1 &\mbox{if } k\geq 1\end{cases}
\end{equation}.

It is interesting to also note that $T_k(x)$ are simply given by the formula
$$
T_k(\cos\theta)=\cos(k\theta),\quad k\in \mathbb N,\quad \theta\in [0,\pi].
$$
We will denote by $\tilde T_k$ the normalized Chebyshev polynomial
$\tilde T_k=\sqrt{\frac{2}{c_k\pi}}T_k$ and the $\tilde T_k$'s then form an orthonormal basis of $L^2([-1,1],\, \mbox{d}\mu)$.

The following lemma gives us an explicit formula for the weighted Finite Fourier transform of $T_k$,
that we failed to find in the literature.

\begin{lemma}
For any $k\in \mathbb N,$ $\widehat T_k$, the weighted finite Fourier transform of $T_k$ is given by
\begin{equation}
\label{Finite_Fourier_Tn}
\widehat T_k(x)=\int_{-1}^1 e^{ix y} T_k(y) \frac{1}{\sqrt{1-y^2}}\,\mathrm{d}y = i^k \frac{\pi}{2} J_k(x).
\end{equation}
\end{lemma}

\begin{proof} This results follows  directly from the formula
$$
\int_{-1}^1\frac{f(y)T_k(y)}{\sqrt{1-y^2}}\,\mathrm{d}y =\int_0^\pi f(\cos\theta)\cos k\theta\,\mbox{d}\theta
$$
applied to $f(y)=e^{ixy}$ and the Poisson integral representation formula of the Bessel function.
\end{proof}

For $f\in L^2([-1,1],\mbox{d}\mu)$ we now define
$$
\tt_nf=\sum_{k=0}^n\scal{f,\tilde T_k}\tilde T_k
$$
the projection of $f$ on $\C_n[X]$ the subspace of $L^2([-1,1],\mbox{d}\mu)$ consisting of polynomials
of degree $\leq n$.
We can now prove the Chebyshev version of Lemma \ref{lem:decay_Bc} and the approximation rate of band-limited functions
by their projection on the Chebyshev orthonomal basis in $L^2([-1,1]\,\mbox{d}\mu)$. However, note that an $L^2(\R)$ function
restricted to $[-1,1]$ need not be in $L^2([-1,1]\,\mbox{d}\mu)$. Therefore, its expansion in the Chebyshev system need
not converge (and not even be defined). Thus, we cannot extend Theorem \ref{th:leg} to the Chebyshev setting.

\begin{proposition}
\label{lem:decay_Bc2}
Let $c>0,$ then for any $f\in \bb_c$, and any $k\geq 0$
\begin{equation}\label{eq:4.1}
|\scal{f, T_k}_{L^2([-1,1],\mbox{d}\mu)}|\leq \frac{1}{\sqrt{(2k+1)c}}\left(\frac{ec}{2(k+1)}\right)^{k+1} \| f\|_{L^2(\mathbb R)},
\end{equation}
and, if $N\geq ec/2$,
$$
\norm{f-\tt_Nf}_{L^2([-1,1],\mathrm{d}\mu)}\leq \frac{e \sqrt{c}}{2(2N+3)}\left(\frac{ce}{2N+4}\right)^{N+1}\| f\|_{L^2(\mathbb R)}^2.
$$
\end{proposition}

\begin{proof}
Since  $f\in \bb_c$, then the Fourier inversion theorem implies that, for $x\in \R$, we have
$$
f(x)= \frac{1}{\sqrt{2\pi}}\int_{-c}^c \widehat{f}(\xi) e^{i\, x\, \xi}\, \mbox{d}\xi
= \frac{c}{\sqrt{2\pi}}\int_{-1}^1 \widehat{f}(c\eta) e^{i\, c\, x\, \eta}\, \mbox{d}\eta.
$$
Combining this with \eqref{eq:2.2.9},  one gets
\begin{eqnarray*}
\scal{f,T_k}_{L^2([-1,1],\mbox{d}\mu)}&=&\int_{-1}^1 f(x)\overline{T_k(x)}\,\frac{\mbox{d}x}{\sqrt{1-x^2}}\\
&=&\frac{c}{\sqrt{2\pi}} \int_{-1}^1  \widehat{f}(c\eta)
\left(\int_{-1}^1e^{-i c x\eta} \overline{T_k(x)}\,\frac{\mbox{d}x}{\sqrt{1-x^2}}\right)
\, \mbox{d}\eta\\
&=& i^k\frac{c\sqrt{2\pi}}{4}\int_{-1}^1  \frac{\pi}{2} J_k(x)(c\eta) \widehat{f}(c\eta)\, d\eta.
\end{eqnarray*}

Using \eqref{eq:bessel} together with  Cauchy-Schwarz inequality and a change of variable, one gets 
\begin{eqnarray*}
|\scal{f,T_k}_{L^2([-1,1])}|&\leq & c^{k+1}\frac{e^{k+1}}{2^{k+2}(k+1)^{k+1/2}}
\int_{-1}^1 |\eta|^k |\widehat{f}(c\eta)|\, d\eta \\
&\leq & c^{k+1/2}\frac{e^{k+1}}{2^{k+3/2}(k+1)^{k+1}}\sqrt{\frac{2}{2k+1}}
\left(\int_{-c}^c  |\widehat{f}(\eta)|^2\, d\eta\right)^{1/2} \\
\end{eqnarray*}
To conclude, it suffices to use Parseval's identity.

\smallskip

From the orthonormality of the $\tilde T_k$'s and this bound, we deduce that
\begin{eqnarray*}
\norm{f-\tt_Nf}_{L^2([-1,1],\mbox{d}\mu)}^2&\leq&\sum_{k=N+1}^{+\infty}|\langle f,\tilde T_k\rangle_{L^2([-1,1],\mbox{d}\mu)}|^2\\
&\leq&\sum_{k=N+1}^{+\infty} \frac{1}{2k+1} c^{2k+1}\frac{e^{2k+2}}{2^{2k+3}(k+1)^{2k+2}}\| f\|_{L^2(\mathbb R)}^2\\
&\leq&\frac{1}{2N+3}\frac{e}{2N+4}\sum_{k=N+1}^{+\infty}\left(\frac{ce}{2(k+1)}\right)^{2k+1}\| f\|_{L^2(\mathbb R)}^2\\
&\leq&\frac{e^2 c}{4}\frac{1}{(2N+3)^2}\left(\frac{ec}{2N+4}\right)^{2N+4}\| f\|_{L^2(\mathbb R)}^2
\end{eqnarray*}
provided $N\geq ec/2$.
\end{proof}

\section{Applications and numerical results}

In the first part of this last section, we apply the quality of approximation of $c-$bandlimited functions by Legendre polynomials
in the framework of prolate spheroidal wave functions (PSWFs).  
As a consequence, we give the convergence rate of the Flammer's scheme, see \cite{Flammer} for the computation of the PSWFs.
 
\subsection{Approximation of prolate spheroidal wave functions}

For  a given  real number $c>0,$ called bandwidth, the Prolate spheroidal wave functions (PSWFs)
denoted by $(\psi_{n,c}(\cdot))_{n\geq 0}$, are defined as  the bounded eigenfunctions  of
the Sturm-Liouville differential operator $\mathcal L_c,$ defined on
$C^2([-1,1]),$ by
\begin{equation}
\label{eq1.0}
\mathcal{L}_c(\psi)=-(1-x^2)\frac{d^2\psi}{d\, x^2}+2 x\frac{d\psi}{d\,x}+c^2x^2\psi.
\end{equation}
They are also  the eigenfunctions of the finite  Fourier transform $\mathcal F_c$, as well as the ones of the operator
$\displaystyle \mathcal Q_c= \frac{c}{2\pi} \mathcal F^*_c \mathcal F_c ,$ which are defined  on $L^2([-1,1])$ by
\begin{equation}
\mathcal{F}_c(f)(x)= \int_{-1}^1
e^{i\, c\, x\, y} f(y)\,\mbox{d}y,
\quad\mbox{and}\quad
\mathcal{Q}_c(f)(x)=\int_{-1}^1 \frac{\sin(c(x-y))}{\pi (x-y)} f(y)\,\mbox{d}y.
\end{equation}
They are normalized so that their $L^2([-1,1])$ norm is equal to $1$ and $\psi_{n,c}(1)>0$. We call  $(\chi_n(c))_{n\geq 0}$ the
corresponding eigenvalues  of $\mathcal L_c$, $\mu_n(c)$ the eigenvalues of $\mathcal{F}_c$ and $\lambda_n(c)$ the ones of
$\mathcal{Q}_c$. A well known property is then that $\norm{\psi_{n,c}}_{L^2(\R)}=\frac{1}{\sqrt{\lambda_n(c)}}$.

The crucial commuting property
of $\mathcal L_c$ and $Q_c$ has  been first observed by Slepian and co-authors \cite{prolate1}, whose name is closely associated
to all properties of PSWFs and their associated spectrum.
Among their basic properties we cite their analytic extension to the whole real
line and their unique properties  to form an orthonormal basis of
$L^2([-1,1])$  and an
orthonormal basis of $\B_c$.
A well known estimate for $\chi_n(c)$ is
\begin{equation}
\label{eq:chinc}
n(n+1)\leq\chi_n(c)\leq n(n+1)+c^2.
\end{equation}
Recall that $\lambda_n(c)$ and $\mu_n(c)$ are related by
$\lambda_n(c)=\dst\frac{c}{2\pi}|\mu_n(c)|^2$.
A precise asymptotic of $\lambda_n(c)$ has been established by Widom \cite{Widom}. 
Recently in \cite{BK1}, the authors have given an explicit approximation of the $\lambda_n(c),$ valid for $n > 2c/\pi$ that 
gives rise to the exact super-exponential decay rate of the sequence of these eigenvalues. But,  here we want a lower bound that is valid
for all $n$. According to \cite{prolate3},
\begin{equation}
\label{eq:land}
0<\lambda_n(c)<1\quad\mbox{and}\quad \lambda_{\ent{\frac{2}{\pi}c}+1}<\frac{1}{2}<\lambda_{\ent{\frac{2}{\pi}c}-1}
\end{equation}
while Bonami-Karoui established the following bound, see \cite{BKnote}
\begin{equation}
\label{eq:bk}
\lambda_n(c)\geq \frac{2}{5}\left(\frac{2c}{\pi(n+1)}\right)^{5(n+1)}\quad\mbox{for }n\geq\max\left(3,\frac{2}{\pi}c\right).
\end{equation}

In Appendix C we will prove the following slight improvement of this bound:

\begin{proposition}\label{prop:Naz}
Let $c$ be a real number. Then, if $\dst n>\frac{2}{\pi}c$,
$$
\lambda_n(c)\geq7\left(1-\frac{2c}{n\pi}\right)^2\left(\frac{c}{7\pi n}\right)^{2n-1}.
$$
If $\dst n=\ent{\frac{2}{\pi}c}$, $\lambda_n(c)\geq\dst\frac{4}{\pi+2c}$.
\end{proposition}

Since $\psi_{n,c}\in L^2([-1,1])$, we may expand it in the Legendre basis
$$
\psi_{n,c}=\sum_{k=0}^{+\infty} \scal{\psi_{n,c},\tilde P_k} \tilde P_k
=\sum_{k=0}^{+\infty} \left(k+\frac{1}{2}\right)\scal{\psi_{n,c},P_k} P_k.
$$

\begin{notation}
Let us write $\beta_k^n(c)=\left(k+\frac{1}{2}\right)\scal{\psi_{n,c},P_k}$ so that, on $[-1,1]$,
\begin{equation}
\label{eq:defbetakn}
\psi_{n,c}=\sum_{k=0}^{+\infty}\beta_k^n(c)P_k.
\end{equation}
\end{notation}

Rokhlin, Xiao and Yarvin \cite{Xiao} have obtained induction formulas for the $\beta_k^n(c)$'s in order to compute
the $\psi_{n,c}$'s. Let us now obtain an estimate for them:

\begin{corollary}
\label{cor:estbetakn}
With the above notation, we have
$$
|\beta_k^n(c)|\leq
\begin{cases}
2\sqrt{\frac{e}{\pi c}}{\sqrt{2k+1}} \left(\frac{ec}{2k+3}\right)^{k+1}
&\mbox{if }n\leq\ent{\frac{2}{\pi}c}-1\\
\sqrt{\frac{e(\pi+2c)}{2\pi c}} {\sqrt{k+1/2}}\left(\frac{ec}{2k+3}\right)^{k+1}
&\mbox{if }n=\ent{\frac{2}{\pi}c}\\
\sqrt{\frac{2 e}{7\pi}}\frac{1}{\sqrt{c}}\left(1-\frac{2c}{n\pi}\right)^{-1}\left(\frac{7\pi n}{c}\right)^{n-1/2}{\sqrt{k+1/2}}\left(\frac{ec}{2k+3}\right)^{k+1}
&\mbox{if }n\geq\ent{\frac{2}{\pi}c}+1
\end{cases}.
$$
\end{corollary}

\begin{proof}
Since $\psi_{n,c}\in\bb_c(\R)$, from Lemma \ref{lem:decay_Bc}
we deduce that
$$
|\scal{\psi_{n,c},P_k}|\leq
2\sqrt{\frac{e}{\pi c}} \frac{1}{\sqrt{2k+1}}\left(\frac{ec}{2k+3}\right)^{k+1} \|\psi_{n,c}\|_{L^2(\mathbb R)}
=2\sqrt{\frac{e}{\pi c}}\frac{1}{\sqrt{2k+1}} \left(\frac{ec}{2k+3}\right)^{k+1}\frac{1}{\sqrt{\lambda_n(c)}}
$$
To conclude, it suffices to use the lower bounds of $\lambda_n(c)$ given by \eqref{eq:land} and the previous proposition.
\end{proof}

From this, one can then easily obtain error estimates for the approximation of prolate spheroidal wave functions by the truncation of their expansion in the Legendre basis in the spirit of Theorem \ref{th:leg}.

\subsection{Numerical results}

In this paragraph, we give several examples that illustrate the different results of this work.\\

\noindent{\bf Example 1.} In this example, we check numerically that the actual error of the uniform approximation of the kernel
${\displaystyle k_n(x,y)=\sum_{k=0}^n h_k(x) h_k(y)}$ may be much smaller than the theoretical error given by Theorem \ref{th:estkn}.
For this purpose, we have considered the value $T=1$ and various values of the integer $10 \leq n\leq 100$. For each value of $n$,
we have used a uniform  discretization $\Lambda$ of the square $[-1,1]^2$ with equidistant 6400 nodes.
Then, we have computed over these grid points, a highly accurate approximation
${\displaystyle \widetilde E_n=\sup_{x,y\in \Lambda} \left|k_n(x,y)-\frac{\sin N(x-y)}{\pi(x-y)}\right|}$ of the exact uniform error
${\displaystyle E_n=\sup_{x,y\in [-1,1]} \left|k_n(x,y)-\frac{\sin N(x-y)}{\pi(x-y)}\right|}$.
The obtained results are given by Table \ref{tab:1}.

\begin{center}
\begin{table}[ht]
\begin{tabular}{|c|c|c|c|c|c|}
\hline
$n$&10&25&50&75&100\\
\hline
$\widetilde E_n$ &0.067& 0.039& 0.025 & 0.023& 0.022  \\
\hline
\end{tabular}
\caption{Approximate  errors  $\widetilde E_n$ for various values of $n$.}
\label{tab:1}
\end{table}
\end{center}

\noindent{\bf Example 2.} In this example, we illustrate the quality of approximation by scaled Hermite functions of a
time-limited and an almost band-limited function. For this purpose, we consider the function $f(x)= 1_{[-1/2,1/2]}(x)$.
From the Fourier transform of $f$, one can easily check that
$f\in H^s(\mathbb R)$ for any $s<1/2$. Note that $f$ is $0$-concentrated in $[-1/2,1/2]$ and since $f\in H^s(\mathbb R)$,
$f$ is $\epsilon_\Omega$-band concentrated in $[-\Omega, +\Omega]$,
with $\epsilon_\Omega < M_s \Omega^{-s}$ with $M_s$ a positive constant. We have considered the value of $c=100$
and we have used \eqref{scaled_approx} to compute the scaled Hermite approximations $K_n^c(f)$ of $f$ with $n=40$ and $n=80$.
The graphs of $f$ and its
scaled Hermite approximation are given by Figure \ref{fig0}. In Figure \ref{fig1}, we have given the approximation errors
$f(x)-K_n^c f(x).$\\
Also, to illustrate the fact that the scaled Hermite approximation outperforms the usual Hermite approximation,
we have repeated the previous numerical tests without the scaling factor
(this corresponds to the special case of $c=1$). Figure \ref{fig2} shows the graphs of $f$ and $K_n f$. This clearly
illustrates the out-performance of the scaled Hermite approximation, compared to the usual Hermite approximation.

\begin{center}
\begin{figure}[ht]
\includegraphics[width=13.5cm]{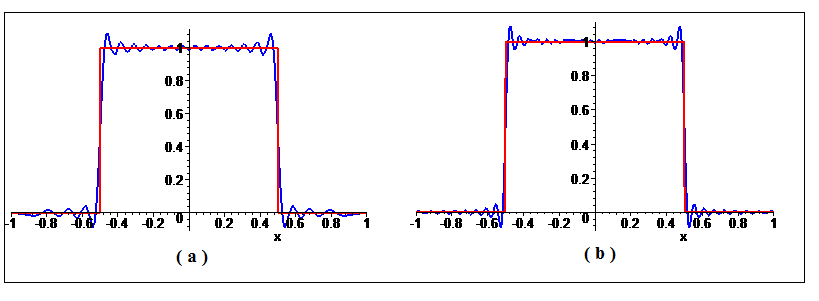}
\caption{Graph of the approximation of $f(x)$ (red) by $K_n^{c} f(x)$, $c=100$ (blue) with
(a) $n=40$  (b) $n=80$.}
\label{fig0}
\end{figure}

\begin{figure}[ht]
\includegraphics[width=13.5cm]{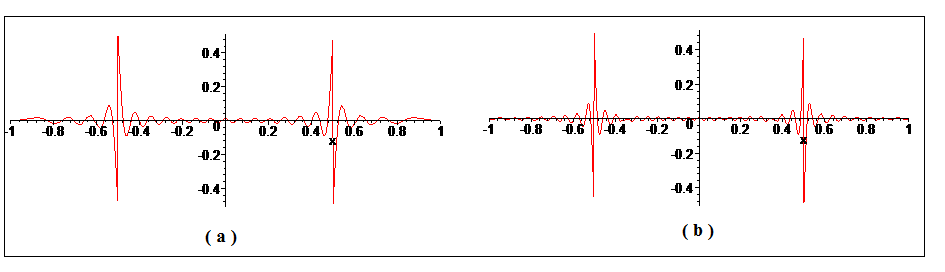}
\caption{Graph of the approximation error $f(x)-K_n^{c} f(x)$, $c=100$ with (a) $n=40$  (b) $n=80$.}
\label{fig1}
\end{figure}
	
\begin{figure}[ht]
\includegraphics[width=13.5cm]{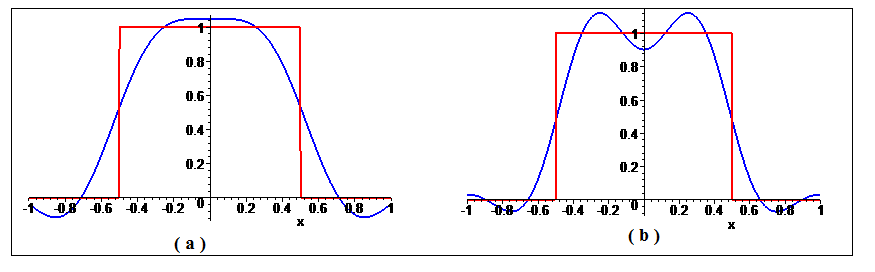}
\caption{Graph of the approximation of $f(x)$ (red) by $K_n^{c} f(x)$, $c=1$ (blue) with (a) $n=40$  (b) $n=80$.}
\label{fig2}
\end{figure}
\end{center}

\noindent {\bf Example 3.} In this  example, we illustrate the decay rate of the Legendre and Chebyshev expansion coefficients of a
$c$-bandlimited functions, that we have given by \eqref{eq:2.1} and \eqref{eq:4.1}, respectively. For this purpose,
we have considered the function $f\in B_c,$ given by
${\displaystyle f_c(x)=\frac{\sin(cx)}{cx},\, x\in \mathbb R}$.
Then, we have computed the different Legendre and Chebyshev expansion coefficients
$l_n(f)=\scal{f, \overline{P_n}}_{L^2(I)}$ and  $c_n(f)=\scal{f, \overline{T_n}}_{L^2(I, d\mu)}$of $f_c,$
for the two values of $c=10$ and $c=50$.  In Figure \ref{fig:1}, we plot the graphs of
the $\log(|l_n(f)|), \, \log(|c_n(f)|),\, n\geq \left[\frac{ec}{2}\right]+1$ versus the logarithm of their respective error bounds given by
\eqref{eq:2.1} and \eqref{eq:4.1}.

\begin{center}
\begin{figure}[ht]
{\includegraphics[width=15cm,height=10.2cm]{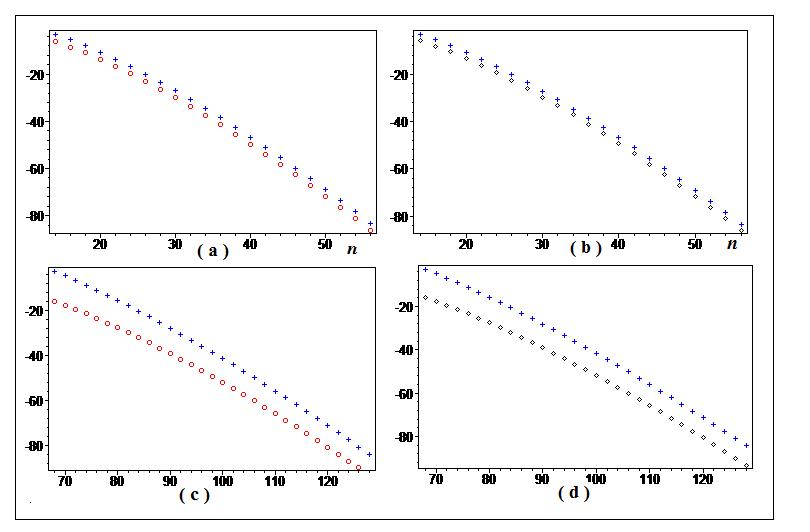}}
\caption{(a) graph of $\log(|l_n(f_c)|)$, $c=10$ (in red) versus the logarithm of its bound \eqref{eq:2.1} (in blue),
(b) graph of $\log(|c_n(f_c)|)$, $c=10$ (in black) versus the logarithm of its bound \eqref{eq:4.1} (in blue);
(c) and (d) same as in (a) and (b) with $c=50$.}
\label{fig:1}
\end{figure}
\end{center}

\noindent {\bf Example 4.}
 In this last  example, we illustrate the quality of approximation by Legendre and Chebyshev polynomials in the Sobolev spaces
$H^s(I).$ We have considered the two  functions $f, g$ given by $f(x)= 1_{[-1/2,1/2]}(x)$ and  $g(x)= (1-|x|) 1_{[-1,1]}(x).$
It is clear that $g\in H^s(I),\, \forall s<3/2.$  In Figure 5, we plot the graphs
of the approximation error of $f$ by its corresponding projections  $\mathcal L_N f$ and $\mathcal T_N f$
over the subspaces spanned by the first $N+1$ Legendre and Chebyshev polynomials, respectively,  with $N=50.$
In Figure 6, we plot the graphs of $g-\mathcal L_N g$ and $g-\mathcal T_N g$ with $N=50$.
\begin{center}
\begin{figure}[ht]
\includegraphics[width=14cm,height=5cm]{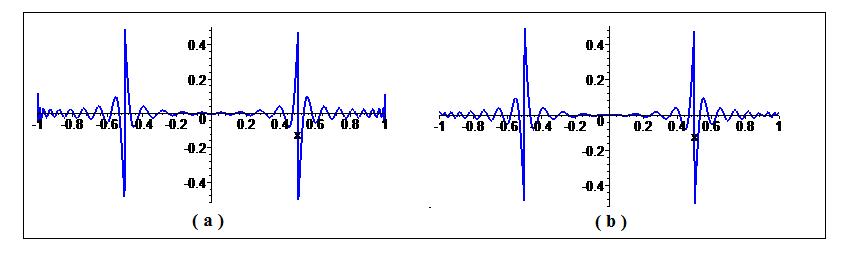}
\caption{(a) Graphs of the errors $f(x)-\mathcal L_N f(x),$ with $N=50$ (b) same as (a) with $f(x)-\mathcal T_N f(x),$.}
\label{fig:5}
\end{figure}

\begin{figure}[ht]
\includegraphics[width=14cm,height=5cm]{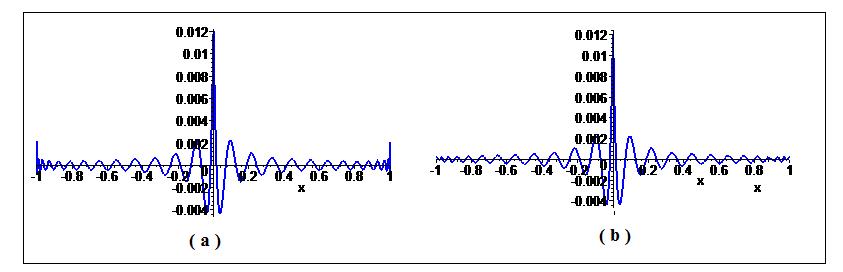}
\caption{(a) Graphs of the errors $g(x)-\mathcal L_N g(x),$ with $N=50$   (b) same as (a) with $g(x)-\mathcal T_N g(x),$.}
\label{fig:6}
\end{figure}
\end{center}

\section*{Acknowledgements}
The first author kindly acknowledge financial support from the French ANR programs ANR
2011 BS01 007 01 (GeMeCod), ANR-12-BS01-0001 (Aventures).
This study has been carried out with financial support from the French State, managed
by the French National Research Agency (ANR) in the frame of the ”Investments for
the future” Programme IdEx Bordeaux - CPU (ANR-10-IDEX-03-02).

The last author kindly acknowledges the financial support form  Michigan State University and from NSF CAREER grant, Award No. DMS-1056965.

\appendix

\section{Proof of Theorem  \ref{th:approxherm}}.

 We will again drop the index $n$ and use the notation introduced before the statement of the theorem.

The bounds for $e(x)$ are obtained by standard calculus, we will thus omit the proof.
As for $E(x)$, the computation shows
$$
E(x)=\frac{1}{\sqrt{p(x)}}\int_0^x\frac{q(t)}{\sqrt{p(t)}}h(t)\sin\bigl(\ffi(x)-\ffi(t)\bigr)\,\mbox{d}t.
$$
Using Cauchy-Schwarz, we obtain
$$
|E(x)|\leq \frac{1}{\sqrt{p(x)}}\left(\int_0^x\frac{q(t)^2}{p(t)}\,\mbox{d}t\right)^{1/2}
\left(\int_0^x h(t)^2\,\mbox{d}t\right)^{1/2}
\leq
\frac{1}{\sqrt{p(x)}}\left(\int_0^x\frac{25\lambda^4}{16p(t)^9}\,\mbox{d}t\right)^{1/2}
$$
since $\norm{h_n}_2=1$.
As $|x|<\lambda$, and $p$ decreases, the estime $|E(x)|\leq \dst\frac{5\lambda^{5/2}}{4p(x)^5}$ follows.
When $|x|\leq\lambda/2$, the change of variable $y=\lambda s$ and a numerical computation
shows that $|E(x)|\leq \frac{2}{\lambda^3}$.

Note that this bound on $E$ directly leads to a bound on $h$. For instance, if $n\geq2$ is even, then
$|h_{2n}(x)|\leq\dst\frac{1}{\sqrt{p(x)}}$ for $|x|\leq\lambda/2$.

The Lipschitz bound on $E$ is a bit more subtle so let us give more details. First, we introduce some further notation:
$$
\chi(x,t)=\frac{q(t)}{\sqrt{p(t)}}h(t)\sin\bigl(\ffi(x)-\ffi(t)\bigr)
\quad\mbox{and}\quad
\Phi(x,y)=\int_0^x\chi(y,t)\,\mbox{d}t.
$$
Now, write
\begin{eqnarray*}
E(y)-E(x)&=&\left(\frac{1}{\sqrt{p(y)}}-\frac{1}{\sqrt{p(x)}}\right)\Phi(y,y)
+\frac{1}{\sqrt{p(x)}}\bigl[\Phi(y,y)-\Phi(x,y)\bigr]\\
&&\frac{1}{\sqrt{p(x)}}\bigl[ \Phi(x,y)-\Phi(x,x)\bigr]=E_1+E_2+E_3.
\end{eqnarray*}
We have proved that for $|y|<\lambda/2$, $\Phi(y,y)\leq 2\lambda^{-3}$.
Simple calculus then implies that $|E_1|\leq\dst\frac{|x-y|}{\lambda^{9/2}}$
when $|x|,|y|<\lambda/2$.

\smallskip

Next, if $|x|,|y|\leq(1-\eta)\lambda$ one can estimate $E_2$ as follows:
\begin{eqnarray*}
|\Phi(y,y)-\Phi(x,y)|&\leq&\int_x^y|\chi(y,t)|\,\mbox{d}t
\leq|x-y|\sup_{|t|\leq\lambda/2}\frac{q(t)}{\sqrt{p(t)}}\sup_{|t|\leq\lambda/2}|h(t)|\\
&\leq& \frac{5\lambda^2}{4p(y)^5}|x-y|.
\end{eqnarray*}
Therefore, $|E_2|\leq \dst\frac{3}{\lambda^{7/2}}|x-y|$.

\smallskip

Finally,
\begin{eqnarray*}
\Phi(x,y)-\Phi(x,x)&=&\int_0^x\frac{q(t)}{\sqrt{p(t)}}h(t)\Bigl[\sin\bigl(\ffi(y)-\ffi(t)\bigr)
-\sin\bigl(\ffi(x)-\ffi(t)\bigr)\Bigr]\,\mbox{d}t\\
&=&2\int_0^x\frac{q(t)}{\sqrt{p(t)}}h(t)\cos\frac{\ffi(x)+\ffi(y)-2\ffi(t)}{2}\,\mbox{d}t
\sin\frac{\ffi(y)-\ffi(x)}{2}.
\end{eqnarray*}
The integral is estimated in the same way as we estimated $\Phi(x,x)$, while for $\ffi$
we use the mean value theorem and the fact that $\ffi'=p$. We, thus, get
$|E_3|\leq \dst\frac{2}{\lambda^{5/2}}|x-y|$. The estimate for $E$ follows.

\section{Proof of Theorem \ref{th:estkn}}.

For sake of simplicity, we will only prove the theorem in the case when $n$ is even and write $n=2p$.
Let $\lambda=\sqrt{2n+1}$, $\mu=\sqrt{2n+3}$, $\alpha=\frac{1}{\sqrt{\pi}p^{1/4}}$,
$\beta=\frac{1}{\sqrt{\pi}p^{1/4}}$, $E=(-1)^p\tilde E_{2p}$ and $F=(-1)^p\tilde E_{2p+1}$.
Then, according to Corollary \ref{cor:simpasympt}
$$
\left\{\begin{matrix}h_{2p}(x)&=&(-1)^p\bigl(\frac{1}{\sqrt{\pi}p^{1/4}}\cos\ffi_{2p}(x)+E(x)\bigr)\\
h_{2p+1}(x)&=&(-1)^p\bigl(\frac{1}{\sqrt{\pi}p^{1/4}}\sin\ffi_{2p+1}(x)+F(x)\bigr)
\end{matrix}\right..
$$
Therefore, $h_{2p+1}(x)h_{2p}(y)-h_{2p+1}(y)h_{2p}(x)$ is
\begin{eqnarray*}
&=&
\frac{1}{\pi p^{1/2}}\bigl(\sin\ffi_{2p+1}(x)\cos\ffi_{2p}(y)-\sin\ffi_{2p+1}(y)\cos\ffi_{2p}(x)\bigr)\\
&&+\frac{1}{\sqrt{\pi}p^{1/4}}\bigl(F(x)\cos\ffi_{2p}(y)-F(y)\cos\ffi_{2p}(x)\bigr)\\
&&+\frac{1}{\sqrt{\pi}p^{1/4}}\bigl(\sin\ffi_{2p+1}(x)E(y)-\sin\ffi_{2p+1}(y)E(x)\bigr)\\
&&+F(x)E(y)-F(y)E(x).
\end{eqnarray*}

The last three terms are all of the form
$$
A(x)B(y)-B(x)A(y)=\bigl(A(x)-A(y)\bigr)B(y)+\bigl(B(y)-B(x)\bigr)A(y)
$$
and are thus bounded with the help of the uniform and Lipschitz bounds of $A$ and $B$ by a factor of $|x-y|$.

The first term is the principal one. Let us start by computing
\begin{eqnarray*}
C:=\sin\ffi_{2p+1}(x)\cos\ffi_{2p}(y)&-&\sin\ffi_{2p+1}(y)\cos\ffi_{2p}(x)\\
&=&\frac{1}{2}\bigl[\sin\bigl(\ffi_{2p+1}(x)+\ffi_{2p}(y)\bigr)-\sin\bigl(\ffi_{2p+1}(x)-\ffi_{2p}(y)\bigr)\\
&&\qquad-\sin\bigl(\ffi_{2p+1}(y)+\ffi_{2p}(x)\bigr)+\sin\bigl(\ffi_{2p+1}(y)-\ffi_{2p}(x)\bigr)\bigr]\\
&=&\sin\frac{\ffi_{2p+1}(x)-\ffi_{2p+1}(y)-\ffi_{2p}(x)+\ffi_{2p}(y)}{2}\\
&&\times\cos\frac{\ffi_{2p+1}(x)+\ffi_{2p+1}(y)+\ffi_{2p}(x)+\ffi_{2p}(y)}{2}\\
&&+\sin\frac{\ffi_{2p+1}(y)+\ffi_{2p}(y)-\ffi_{2p}(x)-\ffi_{2p+1}(x)}{2}\\
&&\times\cos\frac{\ffi_{2p+1}(x)-\ffi_{2p}(x)-\ffi_{2p}(y)+\ffi_{2p+1}(y)}{2}\\
&=&S_1C_1+S_2(C_2-1)+S_2.
\end{eqnarray*}
Now, according to \eqref{eq:lipffi},
$$
|S_1C_1|\leq|S_1|\leq\frac{|\ffi_{2p+1}(x)-\ffi_{2p+1}(y)-\ffi_{2p}(x)+\ffi_{2p}(y)|}{2}\leq\frac{3}{2\sqrt{2n+1}}|x-y|.
$$
With \eqref{eq:lipffi2},
$$
|C_2-1|\leq\frac{|\ffi_{2p+1}(x)-\ffi_{2p}(x)-\ffi_{2p}(y)+\ffi_{2p+1}(y)|^2}{2}\leq\frac{25T^2}{2(2n+1)}.
$$
Thus, with \eqref{eq:ffi},
$$
|S_2(C_2-1)|\leq\left(N+\frac{T^2}{\sqrt{2n+1}}\right)|x-y|\frac{25T^2}{2(2n+1)}\leq\frac{16T^2}{\sqrt{2n+1}}|x-y|.
$$
Finally, using again Lemma \ref{lem:simpasympt}, $\sin\bigl(N(y-x)+\eps_n(y,x)\bigr)$ is
\begin{eqnarray*}
&=&\sin N(y-x)+\sin N(y-x)\bigl(\cos\eps_n(y,x)-1\bigr)+\cos N(y-x)\sin\eps_n(x,y)\\
&=&\sin N(y-x)+E_2(x,y),
\end{eqnarray*}
where
$$
|E_2(x,y)|\leq |\eps_n(x,y)|+\frac{|\eps_n(x,y)|^2}{2}\leq \frac{2T^2}{\sqrt{2n+1}}|x-y|.
$$
It remains to group all estimates to get the result.

\section{Proof of Proposition \ref{prop:Naz}}.

Recall that we want to prove that, if $\dst n>\frac{2}{\pi}c$,
$$
\lambda_n(c)\geq7\left(1-\frac{2c}{n\pi}\right)^2\left(\frac{c}{7\pi n}\right)^{2n-1}.
$$
If $\dst n=\ent{\frac{2}{\pi}c}$, $\lambda_n(c)\geq\dst\frac{4}{\pi+2c}$.\\

According to the min-max Theorem, for any $n$-dimensional subspace $V$ of $L^2(\R)$
$$
\lambda_n(c)\geq \min_{f\in V\setminus\{0\}}\frac{\scal{\mathcal{Q}_cf,f}_{L^2(\R)}}{\norm{f}_{L^2(\R)}^2}.
$$
To show the theorem, we consider $V$ to be the space of functions that is constant
on each interval of the form
\[
I_j:=\left(-1+\frac{2j}{n},-1+\frac{2j+2}{n}\right), \quad j= 0, \ldots, n-1.
\]
Take $f\in V$ and write $\dst f=\sum_{j=0}^{n-1} f_j\chi_{I_j}$. Then
$$
\norm{f}_2^2=\dst\frac{2}{n}\sum_{j=0}^{n-1}|f_j|^2.
$$
On the other hand, write
$$
\widehat{f}(\omega):=\ff[f](\omega):=\frac{1}{\sqrt{2\pi}}\int_{\R}f(t)e^{-it\omega}\,\mathrm{d}t
$$
for the Fourier transform and note that
$\mathcal{Q}_c(f)=\mathcal{F}^{-1}\bigl(\mathbf{1}_{[-c,c]}\mathcal{F}\bigr)$.
Parseval's Identity shows that
$$
\scal{\mathcal{Q}_cf,f}_{L^2(\R)}=\int_{-c}^c|\widehat{f}|^2\, d\xi.
$$
But
\begin{eqnarray*}
\widehat{f}(\xi)&=&\frac{1}{\sqrt{2\pi}}\int_{-1}^1f(t)e^{-i\xi t}\,\mathrm{d}t
=\frac{1}{\sqrt{2\pi}}\sum_{j=0}^{n-1} f_j
\int_{-1+\frac{2j}{n}}^{-1+\frac{2j+2}{n}}e^{-i\xi t}\,\mathrm{d}t\\
&=&\frac{1}{\sqrt{2\pi}}\sum_{j=0}^{n-1}
f_j e^{-i\xi\left(-1+\frac{2j+1}{n}\right)} \int_{-\frac{1}{n}}^{\frac{1}{n}}e^{-i\xi  s}\,\mathrm{d}s\\
&=&\frac{1}{\sqrt{2\pi}}\frac{2}{n}\sum_{j=1}^n
f_j e^{-2ij \frac{\xi}{n}}\,e^{i\xi\frac{n-1}{n}}\frac{\sin \frac{\xi}{n}}{\xi/n}.
\end{eqnarray*}
Therefore,
\begin{eqnarray}
\int_{-c}^c|\widehat{f}(\xi)|^2 \,\mathrm{d}\xi
&=&\frac{2}{\pi n^2}\int_{-c}^c\abs{\sum_{j=1}^nf_j e^{-2ij \frac{\xi}{n}}}^2
\abs{\frac{\sin \frac{\xi}{n}}{\xi/n}}^2\,\mbox{d}\xi\nonumber\\
&=&
\frac{1}{\pi n}\int_{-\frac{2c}{n}}^{\frac{2c}{n}} \abs{\sum_{j=1}^n f_j e^{-ij\eta}}^2
\left(\frac{\sin\eta/2}{\eta/2}\right)^2\,\mbox{d}\eta.
\label{eq:prenaz}
\end{eqnarray}
But, if  $n>\dst\frac{2}{\pi}c$ and $|\eta|<2c/n$ then $|\eta/2|<\pi/2$. Now, on $[-\pi/2,\pi/2]$, $\dst\abs{\frac{\sin t}{t}}\geq 1-\frac{2}{\pi}|t|$.
Therefore $\dst\left(\frac{\sin\eta/2}{\eta/2}\right)^2\geq
\left(1-\frac{2c}{n\pi}\right)^2$.
It follows from \eqref{eq:prenaz} that
\begin{eqnarray*}
\int_{-c}^c|\widehat{f}(\xi)|^2 \,\mathrm{d}\xi&\geq&
\left(1-\frac{2c}{n\pi}\right)^2\frac{1}{\pi n}\int_{-\frac{2c}{n}}^{\frac{2c}{n}} \abs{\sum_{j=1}^n f_j e^{-ij\eta}}^2
\,\mbox{d}\eta\\
&\geq& \left(1-\frac{2c}{n\pi}\right)^2\frac{1}{\pi n}\frac{c}{\pi n}
\left(\frac{4\frac{c}{n}}{14\times 2\pi}\right)^{2(n-1)}
\int_{-\pi}^{\pi} \abs{\sum_{j=1}^nf_j e^{-ij\eta}}^2\,\mbox{d}\eta
\end{eqnarray*}
with the help of Turan's Lemma \cite{Na2}.
Therefore
\begin{eqnarray*}
\int_{-c}^c|\widehat{f}(\xi)|^2 \,\mathrm{d}\xi &\geq&
\left(1-\frac{2c}{n\pi}\right)^2 7\frac{c}{7\pi n}
\left(\frac{\frac{c}{n}}{7 \pi}\right)^{2(n-1)}2\pi
\frac{1}{\pi n}\sum_{j=-\frac{n+1}{2}}^{\frac{n-1}{2}}|f_j|^2\\
&=&7\left(1-\frac{2c}{n\pi}\right)^2\left(\frac{c}{7\pi n}\right)^{2n-1}\norm{f}_2^2.
\end{eqnarray*}
The estimate of $\lambda_n(c)$ follows. If $n \leq\dst\frac{2c}{\pi}$ we may now modify the argument starting from \eqref{eq:prenaz}:
\begin{eqnarray}
\int_{-c}^c|\widehat{f}(\xi)|^2 \,\mathrm{d}\xi
&\geq &
\frac{1}{\pi n}\int_{-\pi}^{\pi} \abs{\sum_{j=1}^n f_j e^{-ij\eta}}^2
\left(\frac{\sin\eta/2}{\eta/2}\right)^2\,\mbox{d}\eta\\
&\geq &
\frac{1}{\pi n}\int_{-\pi}^{\pi} \abs{\sum_{j=1}^n f_j e^{-ij\eta}}^2
\left(1-\frac{|\eta|}{\pi}\right)^2\,\mbox{d}\eta
\end{eqnarray}
since $\dst\abs{\frac{\sin t}{t}}\geq 1-\frac{2|t|}{\pi}$ on $[-\pi/2,\pi/2]$.
But, for $\ell\in\Z$,
$$
\int_{-\pi}^{\pi}\left(1-\frac{|\eta|}{\pi}\right)^2e^{-i\ell\eta}\,\mbox{d}\eta
=\begin{cases}
\frac{2\pi}{3}&\mbox{if }\ell=0\\
\frac{4}{\pi\ell^2}&\mbox{if }\ell\not=0
\end{cases}.
$$
Therefore, using Parseval's equality, one gets
\begin{eqnarray*}
\int_{-c}^c|\widehat{f}(\xi)|^2 \,\mathrm{d}\xi
&\geq &\frac{2}{ n}\frac{1}{3}\int_{-\pi}^{\pi} \abs{\sum_{j=1}^n f_j e^{-ij\eta}}^2\left(1+\frac{6}{\pi^2}\sum_{\ell=1}^n
\frac{\cos\ell \eta}{\ell^2}\right)\,\mbox{d}\eta\\
&\geq &\frac{2}{ n}\frac{1}{3}\int_{-\pi}^{\pi} \abs{\sum_{j=1}^n f_j e^{-ij\eta}}^2\left(1-\frac{6}{\pi^2}\sum_{\ell=1}^n
\frac{1}{\ell^2}\right)\,\mbox{d}\eta\\
&=&\frac{2}{ n}\frac{1}{2\pi}\int_{-\pi}^{\pi} \abs{\sum_{j=1}^n f_j e^{-ij\eta}}^2\,\mbox{d}\eta
\frac{4}{\pi}\sum_{\ell=n+1}^{\infty}
\frac{1}{\ell^2})\\
&=&\left(\frac{4}{\pi}\sum_{\ell=n+1}^{\infty}
\frac{1}{\ell^2}\right)\norm{f}^2
\geq\frac{4}{\pi}\int_{n+1}^{\infty}\frac{\mbox{dx}}{x^2}\norm{f}^2\\
&\geq&\frac{4}{\pi(n+1)}\norm{f}^2.
\end{eqnarray*}
Therefore, for $n\leq\frac{2}{\pi}c$, $\lambda_n(c)\geq\dst\frac{4}{\pi(n+1)}\geq\frac{4}{\pi+2c}$.

\end{document}